\documentclass[11pt] {article}      %

\usepackage{amsmath,amssymb, amsthm, eucal, accents}  
\usepackage{latexsym, graphicx}

\usepackage{txfonts}
\usepackage[T1]{fontenc}     

\usepackage{enumerate}       
\usepackage{float}                 

\usepackage{tikz}
\usetikzlibrary{matrix, arrows, calc}

\theoremstyle{plain}
\newtheorem{theorem}{Theorem}[section]            
\newtheorem{proposition}[theorem]{Proposition}  

\newtheorem{corollary}[theorem]{Corollary}	      %

\theoremstyle{definition}
\newtheorem{definition}[theorem]{Definition}

\numberwithin{theorem}{section}
\numberwithin{equation}{section}
\numberwithin{figure}{section}
\newcommand{\gaction}[2]{\genfrac{}{}{0.5pt}{}{#1}{#2}%
                        \!\lower2pt\hbox{\rotatebox[origin=c]{-90}{{$\looparrowright$}}}}
\newcommand{\dotfraction}[2]{\genfrac{}{}{0.5pt}{}{#1}{#2}%
                        \!\lower.5pt\hbox{{$\circ$}}}

\def\ontop{\accentset}

\begin{document}

\title{\bf Lucas, Fibonacci, and Chebyshev polynomials from matrices}

\author{Jerzy Kocik
\\ \small Department of Mathematics, Southern Illinois University, Carbondale, IL62901
\\\small  jkocik{@}siu.edu  
}

\date{}

\maketitle

\begin{abstract}
\noindent
A simple matrix formulation of the Fibonacci, Lucas, Chebyshev, and Dixon polynomials polynomials is presented.
It utilizes the powers and the symmetric tensor powers of a certain matrix.
\\
\\
{\bf Keywords:} Matrices, Chebyshev polynomials, Fibonacci polynomials, Dixon polynomials, symmetric tensor powers.
\\
\\
{\bf MSC:}  11B39 
\end{abstract}

\def\Tr{{\rm Tr}\,} 
\def\too{\longrightarrow} 
\section{Introduction} 

Chebyshev polynomials play an important role in presentation of many basic mathematical concepts.
They have a number of definitions.  
Here we present a novel definitions of the Chebyshev 
polynomials $T_n(x)$ and $U(x)$ 
in terms of the traces of certain powers of matrices:
$$
T_n(x) = {\small \frac{1}{2}}\;  \Tr \begin{bmatrix}  2x&-1\\ 1& 0\end{bmatrix}^n \,,
\qquad
U_n(x) = {\small \frac{1}{2}}\; \Tr \begin{bmatrix}  2x&-1\\ 1& 0\end{bmatrix}^{\odot n} \,,
$$
where ``$\odot n$'' denotes the $n$-th symmetric tensor power.
Similarly, Lucas and Fibonacci polynomials  $L_n(x)$ and $F_n(x)$ may be interpreted as follows:
$$
L_n(x) = \Tr \begin{bmatrix}  x&1\\ 1& 0\end{bmatrix}^n \,,
\qquad
F_n(x) = \Tr \begin{bmatrix}  x&1\\ 1& 0\end{bmatrix}^{\odot n} \,.
$$
All these polynomials are special cases of the regular Dixon polynomials,
which may also be formulated in terms of matrices:
$$
\ontop 2 D_n(x,y) = \Tr \begin{bmatrix}  x&y\\ 1& 0\end{bmatrix}^n
\qquad
\ontop 1 D_n(x,y) = \Tr \begin{bmatrix}  x&y\\ 1& 0\end{bmatrix}^{\odot n}
$$
We review basic definitions in the next section, and then derive the above formulas.

\newpage

\section{Polynomials: Lucas, Fibonacci, Chebyshev, Dixon}

Chebyshev polynomials emerge from the trigonometric identities, like
$$
\cos(2\varphi) = 2\cos^2(\varphi) - 1 ,
\qquad
\cos(3\varphi) = 4\cos^3(\varphi) - 3\cos^2\varphi
$$
The cosine of a multiple angle is expressible in terms o polynomials in $cos(\varphi)$
due to the recurrence property:
$$
\begin{array}{rl}
\cos((n+1)\varphi)       &  = \cos(n\varphi)\cos\varphi -\sin(n\varphi)\sin\varphi  \\
                                  &  = 2\cos(n\varphi)\cos\varphi - \cos((n-1)\varphi 
\end{array}
$$
Following this induction rule, and a similar one for the function of sine, one defines 
the two types of Chebyshev polynomials presented in Table \ref{table}.

\begin{definition}
Chebyshev polynomials of the first and the second kind, $T(x)$ and $U(x)$, respectively, 
are polynomials in variable $x$ defined in a recursive way as follows:
\begin{equation}
\begin{array}{ll}
T_0 = 1  \,,  \qquad  T_1 = x \,,\qquad \quad T_{n+1} =2x\, T_n - \, T_{n-1}\\
U_0 = 1  \,,  \qquad  U_1 = 2x \,,\qquad U_{n+1} =2x\, U_n - \, U_{n-1}
\end{array}
\end{equation}
\end{definition}

\begin{table}
\hrule

~

$$
\begin{array}{l}
\hbox{Chebyshev first kind}\\[12pt]
\begin{array}{cl}
T_0 &= \ 1  \\
T_1 &= \ x  \\
T_2 &= \ 2x^2 - 1 \\ 
T_3 &= \ 4x^3  -3x  \\
T_4 &= \ 8x^4 -8x^2+1 \\ 
T_5 &= \ 16x^5 -20x^3  + 5x\\
\end{array}\\[41pt]
\quad \boxed{ \ T_{n+1}  = 2x\,T_n - T_{n-1}\phantom{\Big|}\ }\\
\\
\quad \ \cos(n\varphi) = T_n(\cos\varphi) \\
~
\end{array}
\qquad\qquad
\begin{array}{l}
\hbox{Chebyshev second kind}\\[12pt]
\begin{array}{cl}
U_0 &= \ 1  \\
U_1 &= \ 2x  \\
U_2 &= \ 4x^2 - 1 \\ 
U_3 &= \ 8x^3  -4x  \\
U_4 &= \ 16x^4 -12x^2+1 \\ 
U_5 &= \ 32x^5 -32x^3  + 6x\\
\end{array}\\[41pt]
\quad \boxed{ \ U_{n+1}  = 2x\,U_n -  U_{n-1} \phantom{\Big|}\ }\\
\\
\quad\sin(n\varphi) = \dfrac{U_{n-1}(\cos\varphi)}{\sin(\varphi)}\\
\end{array}
$$

$$
\begin{array}{l}
\hbox{Reduced Chebyshev first kind}\\[12pt] 
\begin{array}{cl}
\dot  T_0 &= \ 2  \\
\dot T_1 &= \ x  \\
\dot T_2 &= \ x^2 - 2 \\ 
\dot T_3 &= \ x^3  -3x  \\
\dot T_4 &= \ x^4 -4x^2 +2 \\ 
\dot T_5 &= \ x^5 -5x^3  + 5x\\
\end{array}\\[41pt]
 \quad \boxed{ \ \dot T_{n+1}  = x\,\dot T_n -  \dot T_{n-1}\phantom{\Big|}\ }\\
\\
\quad \ T_n(x) = \frac{1}{2}\dot T(2x)\\
\end{array}
\qquad\quad
\begin{array}{l}
\hbox{Reduced Chebyshev second kind}\\[12pt]
\begin{array}{cl}
\dot  U_0 &= \ 1  \\
\dot  U_1 &= \ x  \\
\dot  U_2 &= \ x^2 - 1 \\ 
\dot  U_3 &= \ x^3  -2x  \\
\dot U_4 &= \ x^4 -3x^2 +1 \\ 
\dot U_5 &= \ x^5 -4x^3  + 3x\\
\end{array}\\[41pt]
\quad \boxed{\  \dot U_{n+1}  = x\,\dot U_n -  \dot U_{n-1}\phantom{\Big|}\ }\\
\\
\quad \ U_n(x) = \dot U(2x)\\
\end{array}
$$

$$
\begin{array}{l}
\hbox{Lucas}\\[12pt]
\begin{array}{cl}
 L_0 &= \ 2  \\
 L_1 &= \ x  \\
 L_2 &= \ x^2 + 2 \\ 
 L_3 &= \ x^3  + 3x  \\
 L_4 &= \ x^4 + 4x^2 +2 \\ 
 L_5 &= \ x^5 +5x^3  + 5x\\
\end{array}\\[41pt]
 \quad \boxed{\  L_{n+1}  = x\, L_n +   L_{n-1}\phantom{\Big|} \ } \\
\end{array}
%
\qquad\qquad
\begin{array}{l} 
\hbox{Fibonacci}\\[12pt]
\begin{array}{cl}
  F_0 &= \ 1  \\
  F_1 &= \ x  \\
  F_2 &= \ x^2 + 1 \\ 
  F_3 &= \ x^3 + 2x  \\
  F_4 &= \ x^4 +3x^2 +1 \\ 
  F_5 &= \ x^5 +4x^3  + 3x\\
\end{array}\\[41pt]
\quad \boxed{\   F_{n+1}  = x\, F_n +   F_{n-1} \phantom{\Big|} \ }\\
\end{array}
$$

~

\hrule
\caption{basic polynomials and their examples}
\label{table}
\end{table}

To eliminate the powers of 2, we may adjust the definitions of the standard Chebyshev polynomials 
to their ``{\bf reduced versions}'', defined and exemplified in the middle part of Table \ref{table}.
It is a small price to pay, as the references to trigonometric identities are still valid:
$$
\cos(n\varphi) = \frac{1}{2}\dot T_n(2\cos\varphi) 
\qquad
\sin(n\varphi) = \dfrac{\dot U_{n-1}(2\cos\varphi)}{\sin(\varphi)}
$$
Another adjustment concerns the sign in the recursive rules, from minus to plus.  
This leads to polynomials that are known as Lucas and Fibonacci, presented in Table \ref{table}.  
Their evaluations at $x=1$ produce the Lucas and Fibonacci sequences, respectively
\cite{Lucas,HL}.

These four types of polynomials have an obvious generalization introduced in 1897 by Leonard Dickson \cite{Dickson}.
We shall distinguish between the standard and the regular versions, the first being the original Dixon's version. 
The differ by the sign convention.

\begin{definition}
\label{def:standardDixon}
The {\bf standard Dixon polynomials}  $P(x,y)$ are  polynomials in two variables $x$ and $y$, defined in a recursive way as follows:
\begin{equation}
\label{eq:dixon}
P_0 = c  \,,  \quad  P_1 = x \,,\qquad \boxed{ \ P_{n+1} =x\, P_n - y \, P_{n-1}\phantom{\Big|}\ }
\end{equation}
where $c$ is a constant.  If $c=2$ or $c=1$, the polynomials will be called of the first kind and the second kind, respectively.
We shall use notation $\ontop c P$ to indicate the initial constant value $c$ in \eqref{eq:dixon}.
(The minus sign is for historical consistency.)
\end{definition}
Here are the first few standard Dixon polynomials given explicitly:
$$
%
\begin{array}{rl}
 \ &\hbox{first kind} \   (c=2)\\
\ontop 2 P_0 \!\!\!&= \ 2 \\
\ontop 2 P_1\!\!\! &= \ x \\
\ontop 2 P_2\!\!\! &= \ x^2 - 2y \\ 
\ontop 2 P_3 \!\!\!&= \ x^3 - 3xy \\
\ontop 2 P_4 \!\!\!&= \ x^4 - 4x^2y +2y^2 \\ 
\ontop 2 P_5 \!\!\!&= \ x^5 - 5x^3y  + 3xy^2
\end{array}
\quad\qquad
\begin{array}{rl}
\ &\hbox{second kind} \ (c=1)\\
\ontop 1 P_0 \!\!\!&= \ 1  \\
\ontop 1 P_1 \!\!\!&= \ x \\
\ontop 1 P_2\!\!\! &= \ x^2 - y \\ 
\ontop 1 P_3 \!\!\!&= \ x^3 - 2xy \\
\ontop 1 P_4 \!\!\!&= \ x^4 - 3x^2y +y^2 \\ 
\ontop 1 P_5 \!\!\!&= \ x^5 - 4x^3y  +2xy^2
\end{array}
$$

This is a generalization of the known polynomials and sequences:
\\

\begin{tabular}{r|cc}
$\ontop c P(x,y)$   &$c=2$&$c=1$\\
\hline
\\[-10pt]
\  &  (reduced) &(reduced)  \\
$y=1 \ \too$ & Chebyshev polynomials & Chebyshev polynomials \\
 &  \ of the first second kind& \ of the second kind \\[3pt]
$y=-1 \ \too$ &  Lucas polynomial   &  Fibonacci polynomials \\[3pt]
$y=-1, x=1 \ \too$ &  Lucas numbers  &  Fibonacci numbers  \\
\end{tabular}

~\\
\\
One may find more appealing the altered definition of Dixon polynomials, 
namely with difference in \eqref{eq:dixon} replaced by the sum:

\begin{definition}
The {\bf regular Dixon polynomials} are defined by the following recurrence 
\begin{equation}
\label{eq:dixon+}
D_0 = c  \,,  \quad  D_1 = x \,,\qquad \boxed{ \ D_{n+1} =x\, D_n + y \, D_{n-1}\phantom{\Big|}\ }
\end{equation}
where $c$ is a constant (cf. Definition \ref{def:standardDixon}.
\end{definition}

This is of course equivalent definition, the difference between the regular and standard Dixon polynomials 
is due to replacement $y\to (-y)$, that is,  $P(x,y) = D(x,-y)$.  
The regular Dixon polynomials have all coefficient signs positive:
$$
\begin{array}{rl}
 \ &\hbox{first kind} \   (c=2)\\
\ontop 2 D_0 \!\!\!&= \ 2 \\
\ontop 2 D_1\!\!\! &= \ x \\
\ontop 2 D_2\!\!\! &= \ x^2 + 2y \\ 
\ontop 2 D_3 \!\!\!&= \ x^3 + 3xy \\
\ontop 2 D_4 \!\!\!&= \ x^4 + 4x^2y +2y^2 \\ 
\ontop 2 D_5 \!\!\!&= \ x^5 + 5x^3y  + 3xy^2
\end{array}
\quad\qquad
\begin{array}{rl}
\ &\hbox{second kind} \  (c=1)\\
\ontop 1 D_0 \!\!\!&= \ 1  \\
\ontop 1 D_1 \!\!\!&= \ x \\
\ontop 1 D_2\!\!\! &= \ x^2 + y \\ 
\ontop 1 D_3 \!\!\!&= \ x^3 + 2xy \\
\ontop 1 D_4 \!\!\!&= \ x^4 + 3x^2y +y^2 \\ 
\ontop 1 D_5 \!\!\!&= \ x^5 + 4x^3y  +2xy^2
\end{array}
$$

\newpage

\section{Polynomials via matrices}

It is well-known that $2\times 2$ determinant-1 matrices of the  special linear group  ${\rm SL}(2)$
relate to Chebyshev polynomials, namely 
$$
\Tr M^n = T_n(\Tr M)
$$ 
(Chebyshev polynomial  of the first kind evaluated for $M$ with $\det M=1$). 
A generalization of the above to matrices of arbitrary determinant was offered in \cite{Owczarek}.
We restate it below, using the language of symmetric powers and with adjusted terminology.
(For details on the symmetric tensor powers see \cite{jk-fk}, Sec 7.2., and the Appendix below.)

\begin{proposition}
\label{thm:owczarek}
\label{thm:1}
Let $M\in {\rm Mat(2,\mathbb F)}$ be any $2\times 2$ matrix over some field $\mathbb F$.  
The trace of the $n$-th power and of the symmetric $n$-th power of $M$ 
coincides with the Dixon polynomial of the second and first kind, corresoindingly,
with the variables equal to $x=\Tr M$ and $y=\det M$, i.e., 
\begin{equation}
\label{eq:owczarek}
\begin{array}{lll}
(a) & \Tr ( M^n)   &  = \  \ontop 2 P_n(\Tr M,\; \det M) \\
(b) & \Tr ( M^{\odot n}) & = \  \ontop 1 P_n(\Tr M,\; \det M) 
\end{array}
\end{equation}
\end{proposition}

\noindent
{\bf Proof of the first statement}:
Start with an observation that for any pair of $2\times 2$ matrices 
$A,B\in {\rm Mat}(2, \mathbb F)$ this simple equation holds:
\begin{equation}
\Tr ( A^2B) \  = \ {\Tr}\, A\cdot \Tr  AB   -  {\det}\, A \cdot \Tr B 
\end{equation}
(show by direct inspection).
Substitute $A=M$ and $B=M^n$.
We get
\begin{equation}
\Tr ( M^{n+2}) \  = \ \Tr M\cdot \Tr  M^{n+1}   -  {\det}\, A \cdot \Tr M^n 
\end{equation}
Denote $\Tr M = x$ and $\det M = y$.
Then
$$
\Tr M^{n+2} = x \Tr M^{n+1} - y \Tr M^n
$$
Since the traces satisfy the same rule as the polynomials and start with the  same initial values:
$M^0 ={\rm Id}$ (identity matrix) and $\Tr {\rm Id}=2$, 
as well as $M^1 = M$, thus denoting $x=\Tr M$ we get the result.
\\
\\
{\bf Proof of the second statement:}
Note that the trace is invariant under similarity transformation of matrices, $M\to TMT^{-1}$
and every matrix may be brought to the lower-triangular form, we focus on matrix 
$$
M \ = \ \begin{bmatrix} a&0\\ c&b\end{bmatrix}
$$
We shall show that 
\begin{equation}
\label{eq:aaa}
\Tr M^{\odot n} \ = \ \sum_{i=0}^n  a^{n-i}b^i\ = \  a^n + a^{n-1}b + a^{n-2}b^2 +\ldots + b^n  
\end{equation}
Indeed:
$$ \begin{bmatrix} a&0\\ c&b\end{bmatrix}^T
\begin{bmatrix} x\\ y\end{bmatrix}
=
 \begin{bmatrix} ax\\ cx+by\end{bmatrix}
$$
The k-th basis vector $\mathbf e_k \equiv x^{n-k}y^k$ is transformed to 
$$
\begin{aligned}
(ax)^{n-k}(cx+by)^k & \ = \ 
(ax)^{n-k} \; \sum_i \;{k \choose i} \,c^ix^i\; b^{k-i}y^{k-i}
\\
     & \ = \ \sum_i {k \choose i} \, a^{n-k}b^{k-i} c^i \, {\mathbf e}_{k+i}
\end{aligned}
$$
Thus the coefficient at $\mathbf e_k$ corresponds to $i=0$
and therefore the $k$-th entry on the diagonal of $M^{\odot n}$ is $a^nb^{n-k}$.
The trace is therefore as in \eqref{eq:aaa}.
Now, it is a simple algebraic manipulation to see that:
$$
\begin{aligned}
 a^n + a^{n-1}b + a^{n-2}b^2 +\ldots + b^n  & \ =  \ (a+b)\cdot ( a^{n-1} + a^{n-2}b + a^{n-3}b^2 +\ldots + b^{n-1}) \\
                            \                                          & \  \quad \quad - ab\cdot( a^{n-2} + a^{n-3}b + a^{n-4}b^2 +\ldots + b^{n-2} )
\end{aligned}
$$
This is equivalent to
$$
\Tr M^{\odot n} \ = \ \Tr M \cdot \Tr M^{\odot(n-1)}  - \det M \cdot \Tr M^{\odot(n-2)}
$$
The initial terms agree with \eqref{eq:dixon} for $c=1$; \  indeed, $\Tr M^{\odot 0} \equiv \Tr[1] = 1$
and $\Tr M^{\odot 1} \equiv \Tr M$.
\qed
\\

We are now  ready justify the matrix representations of the Dixon polynomials:

\begin{proposition}
The following may be viewed as the definition of the standard Dixon  polynomials
$$
\ontop 2 P_n(x,y) = \Tr \begin{bmatrix}  x&-y\\ 1& 0\end{bmatrix}^n
\qquad
\ontop 1 P_n(x,y) = \Tr \begin{bmatrix}  x&-y\\ 1& 0\end{bmatrix}^{\odot n}
$$
\end{proposition} 

\begin{proof}
This may be seen as a simple corollary to Proposition \ref{thm:owczarek}
by a smart choice of the matrix.
Namely, observe that the following matrix has the trace and determinant 
matching with the variables $x$  and $y$:
$$
\Tr \begin{bmatrix}  x&-y\\ 1& 0\end{bmatrix} = x \,,
\qquad
\det \begin{bmatrix}  x&\! -y\\ 1& 0\end{bmatrix} = y
$$
Substitute this matrix to \eqref{eq:owczarek}, and the result follows.
\end{proof}

\newpage

\begin{corollary}
In particular, the following are alternative definitions of  Lucas, Fibonacci  and Chebyshev, polynomials:
%
$$
\begin{array}{lll}
\hbox{\rm Dixon regular}
\qquad &
\ontop 2 D_n(x) = \Tr \begin{bmatrix}  x&y\\ 1& 0\end{bmatrix}^n
\qquad &
\ontop 1 D_n(x) = \Tr \begin{bmatrix}  x&y\\ 1& 0\end{bmatrix}^{\odot n}
\\[14pt]
\hbox{\rm Lucas and Fibonacci}
\qquad &
L_n(x) = \Tr \begin{bmatrix}  x&1\\ 1& 0\end{bmatrix}^n
\qquad &
F_n(x) = \Tr \begin{bmatrix}  x&1\\ 1& 0\end{bmatrix}^{\odot n}
\\[14pt]
\hbox{\rm Chebyshev reduced}
\qquad &
\dot T_n(x) = \Tr \begin{bmatrix}  x&-1\\ 1& 0\end{bmatrix}^n
\qquad &
\dot U_n(x) = \Tr \begin{bmatrix}  x&-1\\ 1& 0\end{bmatrix}^{\odot n}
\\[14pt]
\hbox{\rm Chebyshev}
\qquad &
T_n(x) = \Tr \begin{bmatrix}  2x&-1\\ 1& 0\end{bmatrix}^n
\qquad  &
U_n(x) = \Tr \begin{bmatrix}  2x&-1\\ 1& 0\end{bmatrix}^{\odot n}
\end{array}
$$
\end{corollary}

\section*{Appendix}


{\large \bf A.1 Symmetric powers}
\\
\\
Here we present a quick way to calculate the symmetric tensor powers of a $2\times 2$ matrix.
For this purpose, we use $x$ and $y$ as dummy variables (not to be confused with $x$ and $y$ of the Dixon polynomials). 
Let us define the following map from 2- to $(n+1)$-dimensional standard vector space:
$$
\begin{bmatrix}  x\\ y\end{bmatrix} 
                          \quad \longrightarrow \quad 
\begin{bmatrix}  x\\ y\end{bmatrix}^{\odot n} \ = \  
                          [\,x^n,\; x^{n-1}y, \; x^{n-2}y^2,\; ... ,\; xy^{n-1}, \; y^n\,]^T
$$ 
Denote $\mathbf v = [ x, y]^T$.
The $n$-th symmetric power of a $2\times 2$ matrix 
$M$
is an  $(n+1)\times(n+1)$ matrix $M^{\odot n}$ defined by
%
\begin{equation}
\label{eq:symmetric}
\boxed{\quad
 \left(\,M\,\mathbf v\,\right)^{\odot n} \ = \  M^{\odot n}\,\mathbf v^{\odot n} 
\phantom{\Big|^m}\ }
\end{equation}

~

\noindent
{\bf Example:}
Consider the matrix
$$
M = \begin{bmatrix}  a&b\\ 1& 0\end{bmatrix} \,.
$$
The calculations for the second symmetric power are:
$$
M\mathbf v \ =\  \begin{bmatrix}  a&b\\ 1& 0\end{bmatrix} 
                         \begin{bmatrix}  x\\ y\end{bmatrix} 
\ = \ 
\begin{bmatrix}  ax+by\\ x\end{bmatrix} 
\quad \longrightarrow \quad
\begin{bmatrix}  (ax+by)^2\\ (ax+by)x\\ x^2\end{bmatrix} 
\ = \ 
\begin{bmatrix}  a^2x^2+2abxy +b^2y^2 \\ ax^2+bxy\\ x^2\end{bmatrix} 
\ = \ 
\begin{bmatrix}  a^2&2ab&b^2\\
                         a&b& 0\\
                         1&0&0\end{bmatrix} 
\begin{bmatrix}  x^2\\
                         xy\\
                         y^2\end{bmatrix} 
$$   
Thus the second power $M^{\odot 2}$ is the last square matrix, above.
Other powers are calculated similarly from \eqref{eq:symmetric}.

~

\noindent
{\bf Remark:}
If the action of the matrices is defined on the dual space, i.e., the row vectors, the symmetric powers differ slightly,
but the diagonals of both versions coincide.  

~

For other applications of symmetric tensor powers, see \cite{F-K, jk-fk}.

\newpage

Below, we present the regular and the symmetric powers of 
matrix
$$
M = \begin{bmatrix}  x&y\\ 1& 0\end{bmatrix} \,.
$$
The regular powers are:
$$
\begin{aligned}
M^0 & \ = \  \begin{bmatrix}  1&0 \\ 0& 1\end{bmatrix} \\
M^1 & \ = \  \begin{bmatrix}  x&y \\ 1& 0\end{bmatrix} \\
M^2 & \ = \  \begin{bmatrix}  x^2+y&xy   \\ 
                                            x & y     \end{bmatrix}  \\
M^3 & \ = \  \begin{bmatrix}  x^3 +2xy 	& x^2y+y^2   	 \\ 
                                             x^2 +y 	& xy         \end{bmatrix}\\
M^4 & \ = \  \begin{bmatrix}  x^4 +3x^2y +y^2  	&x^3y +2xy^2 \\ 
                                            x^3 +2xy 	&x^2y  +y^2  	     \end{bmatrix}\\
M^5 & \ = \  \begin{bmatrix}  x^5 +3x^3y +3xy^2  	&x^4y +3x^2y^2 +y^3  \\ 
                                            x^4 +3x^2y +y^2 	   	&x^3y  +y^2  	     \end{bmatrix}
\end{aligned}
$$

\noindent
The symmetric tensor powers are:
$$
\begin{aligned}
M^{\odot 0} & \ = \  \begin{bmatrix}  0 \end{bmatrix} \\
M^{\odot 1} & \ = \  \begin{bmatrix}  x&y \\ 1& 0\end{bmatrix} \\
M^{\odot 2} & \ = \  \begin{bmatrix}  x^2&2xy  & y^2  \\ 
                                                      x      & y   & 0  \\
                                                      1   & 0  & 0         \end{bmatrix}  \\
M^{\odot 3} & \ = \  \begin{bmatrix}  x^3  	&3x^2y    	&3xy^2	&y^3   \\ 
                                                      x^2  	& 2xy   	&y^2 	&0   \\
                                                       x    	& y	  	&0		&0    \\
                                                      1   	&  0 		& 0    	&0   \end{bmatrix}\\
M^{\odot 4} & \ = \  \begin{bmatrix}  x^4  	&4x^3y    	&6x^2y^2	&4xy^3  	&y^4   \\ 
                                                      x^3  	&3x^2y    	&3xy^2	&y^3  	&0   \\
                                                    x^2    	&  2xy    	&y^2		&0		&0    \\
                                                      x   	&y		&0		& 0    	&0  \\
						        1   	& 0  		&0		& 0    	&0   \end{bmatrix}
\end{aligned}
$$
One may easily verify that the traces of the powers correspond to the regular Dixon polynomials of the second and first kind,
$\Tr M^n = \ontop 2 D_n$ and $\Tr M^{\odot n} = \ontop 1 D_n$ 
 (and to the standard version, under replacement $y\to -y$). 
Substituting $y=1, -1$ yields the further specifications to the other polynomials discussed: Chebyshev, Lucas and  Fibonacci.
Note that under substitution $x=y=1$,
the symmetric powers become the upper parts of the Pascal triangle in a matrix form. 
\\

\noindent
{\large \bf A.2  A handful of properties}
\\
\\
The matrix formulation of the polynomials 
leads to other interesting properties and symmetries. 
Here we mention just a few:

~

\noindent
(1)

\vspace{-.5in}

$$
M^{n+1} 
\ = \ \begin{bmatrix} \ontop 1D_{n+1} 	& y\, \ontop 1 D_n   	 \\ 
                                \ontop 1D_n 	& y\, \ontop 1 D_{n-1}         \end{bmatrix}\,.
$$
for instance:
$$
M^3 
 \ = \  \begin{bmatrix}  x^3 +2xy 	& x^2y+y^2   	 \\ 
                                             x^2 +y 	& xy         \end{bmatrix}
\ = \ \begin{bmatrix} \ontop 1D_3 	& y\, \ontop 1 D_2   	 \\ 
                                \ontop 1D_2 	& y\, \ontop 1 D_1         \end{bmatrix}\,.
$$

\noindent
(2)

\vspace{-.35in}

$$
\ontop 1 D_n + y\, \ontop 1 D_{n-2}  =  \ontop 2 D_n \,.
$$

\noindent
(3)

\vspace{-.35in}

$$
\det M^{\odot n} = \big(   \det M \big)^{   \frac{n(n+1)}{2}       }  \,.
$$
(4) \ 
Off-diagonal sums:
$$
M^{\odot n}_{1,1+k} + M^{\odot n}_{1,2+k} +...+ M^{\odot n}_{n-k,n} \ = \ y^k\,D_{n-k} \,.
$$
These and other properties will be discussed elsewhere.


\end{document}